\documentclass[a4paper,10pt]{article}

\usepackage{amsfonts}
\usepackage{amssymb}
\usepackage{amsmath}

\newtheorem{theor}{Theorem}[section]
\newtheorem{propo}[theor]{Proposition}
\newtheorem{coroll}[theor]{Corollary}
\newtheorem{lemma}[theor]{Lemma}
\newtheorem{conje}[theor]{Conjecture}
\newtheorem{defi}[theor]{Def\mbox{}inition}
\newtheorem{defis}[theor]{Def\mbox{}initions}

\newenvironment{proof}
  {\noindent {\sc Proof}}
  {\par \hfill \rule{1ex}{1ex}}

\def\sucx{{\bar x}}
\def\sucy{{\bar y}}
\def\sucz{{\bar z}}
\def\asra{{ar}}
\def\asce{{ac}}
\def\ascd#1{{ac_{#1}}}
\def\mach{{cac}}
\def\dmach{{\delta-cac}}
\def\limsupline{\mathop{\overline{\lim}}}
\def\liminfline{\mathop{\underline{\lim}}}
\newcommand{\myatop}[2]{\genfrac{}{}{0pt}{}{#1}{#2}}

\def\adh#1{\overline{#1}}
\def\co{\mathop{\rm co}\nolimits}

\def\cb{{\cal B}}
\def\cc{{\cal C}}
\def\ct{{\cal T}}
\def\cv{{\cal V}}

\def\IN{\mathbb{N}}
\def\IK{\mathbb{K}}
\def\IR{\mathbb{R}}

\newcommand{\knat}{k\in\mathbb{N}}
\newcommand{\mnat}{m\in\mathbb{N}}
\newcommand{\nnat}{n\in\mathbb{N}}

\newcommand{\disp}{\displaystyle}
\newcommand{\vareps}{\varepsilon}
\newcommand{\longto}{\longrightarrow}

\title{Continuity properties of sequentially asymptotically center-complete spaces}
\author{C. Angosto, M. C. List\'an-Garc\'{\i}a, F. Rambla-Barreno}
\date{{\small \textit{Dpto. de Matem\'atica Aplicada y Estad\'{\i}stica, Universidad Polit\'ecnica de Cartagena \footnote{The f\mbox{}irst author was partially supported by the project MTM2011-25377 of
the Spanish Ministry of Science and Innovation. The second and third authors were partially supported by Junta de Andaluc\'{\i}a and FEDER grant FQM-257.} \\ Paseo de Alfonso XIII 52,
30203-Cartagena (Murcia), Spain \\ \textit{email: carlos.angosto@upct.es} \\ \mbox{} \\ Dpto. de Matem\'aticas, Universidad de C\'adiz \\ Apdo. 40, 11510-Puerto Real (C\'adiz), Spain} \\ \textit{emails: mariadelcarmen.listan@uca.es; fernando.rambla@uca.es}}}

\begin{document}

\maketitle

\abstract{We obtain formulae to calculate the asymptotic center and radius of bounded sequences in $\cc_0(L)$ spaces. We also study the existence of continuous selectors for the asymptotic center map in general Banach spaces. In Hilbert spaces, even a H\"older-type estimation is given.}

\section{Introduction}

The notions of Chebyshev center and radius were introduced by A. L. Garkavi (\cite{Ga}) to study some approximation problems in normed spaces:

\begin{defis}
Let $X$ be a normed space. If $A \subseteq X$ is bounded, its {\sl Chebyshev radius} is given by 
$$r(A) = \inf_{y\in X}\sup_{x \in A} \|x-y\|$$
and its Chebyshev center by
$$c(A) = \{y \in X: \sup_{x \in A} \|x-y\| = r(A)\}.$$
\end{defis}

This was followed by M. Edelstein's notions (\cite{Ed}) of asymptotic center and radius of a bounded sequence. These were def\mbox{}ined in uniformly convex spaces and subsequently generalized to Banach spaces by T. C. Lim in \cite{Lim74} (who went even further, dealing with well-ordered nets). We will need an additional concept, related to the asymptotic center:

\begin{defis}
Let $X$ be a normed space. If $\sucx = {(x_n)}_n$ is a bounded sequence in $X$, its {\sl asymptotic radius} is given by 
$$\asra(\sucx) = \inf_{y\in X}\limsupline_n \|x_n-y\|$$
and its asymptotic center by
$$\asce(\sucx) = \{y \in X: \limsupline_n \|x_n-y\| = \asra(\sucx)\}.$$
Let us def\mbox{}ine also the set, depending on $\delta \geq 0$,
$$\ascd{\delta}(\sucx) = \{y \in X: \limsupline_n \|x_n-y\| \leq \asra(\sucx)+\delta\}.$$
\end{defis}

For bounded, decreasing nets of sets, the concepts of ``asymptotic center'' and ``asymptotic radius'' can be def\mbox{}ined analogously (\cite{Lim83}) and generalize both the Chebyshev center / radius of a bounded set and the asymptotic center/radius of a bounded sequence.

We will also say that a Banach space is center-complete $/$ sequentially asymptotically center-complete $/$ asymptotically center-complete (in short, $cc / sacc / acc$) whenever every bounded set $/$ bounded sequence $/$ bounded net of sets has a nonempty center $/$ asymptotic center $/$ asymptotic center.

There exist many results (\cite{AmMaSa}, \cite{Ga}, \cite{Lim83}, \cite{Ve}) concerning the existence of centers and asymptotic centers. Moreover, in the case of center-complete spaces several authors (\cite{AmMa}, \cite{BaPa}, \cite{Ma}) have tried and found conditions guaranteeing the existence of a continuous selector for the center map, i. e. a continuous $\varphi: \cb \longto X$ satisfying $\varphi(A) \in c(A)$, where $cb$ is the set of bounded subsets of the normed space $X$, endowed with the Hausdorf\mbox{}f metric. Let us recall that this selector may fail to exist even in the 3-dimensional case (\cite{AmMa}).

In this paper we provide formulae to calculate the asymptotic center and radius of sequences in $\cc_0(L)$ and also give some results concerning existence of continuous selectors for bounded sequences, in analogy with the aforementioned ones. Note that, in the separable case, each result concerning sequential asymptotic center completeness produces a result on (Chebyshev) center completeness (\cite{LiRa}).

More specifically, in corollary \ref{corollphi} we obtain, for certain Banach spaces, a continuous mapping $\varphi$ such that $\varphi(\sucx) \in \asce(\sucx)$ for every bounded sequence $\sucx$ and, additionally:

\begin{itemize}
\item $\varphi(\sucx)= \lim_n x_n$ if $\sucx$ converges.
\item $\varphi(\sucx) = \varphi(F(\sucx))$ where $F$ is the forward operator.
\item $\varphi(\sucx) = \varphi({(x_{\pi(n)})}_n)$ for every bijection $\pi: \IN \to \IN$.
\end{itemize}

We only deal with real Banach spaces, usually denoted by $X$ or $Y$. The space of bounded sequences in $X$ is denoted by $\ell_\infty(X)$. Every topological space considered is Hausdorf\mbox{}f, and $K$ will always denote a compact space. Similarly, $L$ will always be a locally compact space. The Banach spaces $\cc(K)$ and $\cc_0(L)$ are as usual: the space of continuous functions def\mbox{}ined on $K$ and the space of continuous functions def\mbox{}ined on $L$ and vanishing at inf\mbox{}inity, i. e. those $f: L \to \IK$ continuous and such that for every $\vareps>0$ the set $\{t \in L : |f(t)| \geq \vareps\}$ is compact (note that this includes the $\cc(K)$ spaces as a particular case).

The notions we have studied are ``absolute'' center and radius in all cases. It is also possible to study the ``relative'' versions (e. g. the center of a subset of $\ell_\infty$ with respect to $c_0$), which is also a classical topic and would introduce an additional level of complexity in the problem.

\section{A formula for the asymptotic center in $\cc_0(L)$ spaces}

T. C. Lim (\cite{Lim83}) proved that every $\cc(K)$ is asymptotically center-complete. He also gave formulas to calculate the asymptotic center and radius of every bounded sequence in some spaces, namely $c_0$, $c$ and $\ell_\infty$.

In what follows we will give a generalization of the sequential case by proving that every $\cc_0(L)$ space is sequentially asymptotically center-complete. Moreover, the proof presented here provides a formula for the center and radius in all such spaces. As an example of application, we will show how Lim's formulae for the radius can be retrieved from ours.

We need two lemmas, the f\mbox{}irst one is well known and can be found e.g. in \cite{Na}, p. 442:

\begin{lemma}
Let $\ct$ be a Hausdorf\mbox{}f topological space. Then $\ct$ is normal if and only if for every upper semicontinuous
function $f: \ct \to \IR$ and lower semicontinuous function $h: \ct \to \IR$ satisfying $f\leq h$, there exists a continuous function $g: \ct \to \IR$ such that $f \leq g \leq h$.
\end{lemma}

As a consequence, let us prove

\begin{lemma}\label{ndist}
Let $\ct$ be a normal Hausdorf\mbox{}f space. If $a, b: \ct \to \IR$ are respectively lower and upper semicontinuous functions, $a \leq b$ and $b-a$ is bounded, then there exists a continuous function $g: \ct \to \IR$ such that $b-g$ is bounded and $$\|b-g\|=\|g-a\|=\frac{1}{2}\|b-a\|.$$ Moreover, for every $t_0 \in \ct$ and $s \in [a(t_0),b(t_0)]$ there exists a continuous function $g: \ct \to \IR$ such that $b-g$ is bounded, $g(t_0)=s$ and $$\max\{\|b-g\|,\|g-a\|\}= \max\{b(t_0)-g(t_0), g(t_0)-a(t_0), \frac{1}{2}\|b-a\|\}.$$
\end{lemma}

\begin{proof}

Let $\gamma = \frac{1}{2} \| b - a \|$. Def\mbox{}ine $f,h:\ct \rightarrow \IR$ as $f(t) = b(t) - \gamma$ and $h(t)=a(t) + \gamma$; by the previous lemma, there exists $g: \ct \rightarrow \IR$ continuous and such that $f \leq g \leq h$. For every $t \in \ct$ we have
$$b(t) - g(t) -\gamma \leq 0 \leq a(t) - g(t) + \gamma$$
from this and $a(t) \leq b(t)$, we get
$$a(t) - g(t) \leq b(t) - g(t) \leq \gamma$$ 
and
$$g(t) - b(t) \leq g(t) - a(t) \leq \gamma.$$
Therefore $b - g$ and $g - a$ are bounded and $\max\{ \| b - g \|, \|g - a\| \} \leq \frac{1}{2} \| b - a \|$. Using now the triangle inequality we obtain $\| b - a \| \leq \|b - g\| + \|g - a\| \leq 2 \max\{ \| b - g \|, \|g - a\|\} \leq \|b-a\|$ and this implies
$$\|b - g\|=\|g - a\|=\frac{1}{2}\|b - a\|.$$

Now let $t_0 \in \ct$ and $s \in [a(t_0),b(t_0)]$. Denote by $\chi_{\{t_0\}}$ the characteristic function of $\{t_0\}$. We can assume without loss of generality that $b(t_0)-s \leq s-a(t_0)$ and so $2s \geq b(t_0)+a(t_0)$. Consider $$\tilde{b} = b + (\max\{2s-a(t_0)-b(t_0),\gamma+s-b(t_0)\})\chi_{\{t_0\}}$$ and $$\tilde{a} = a - (\max\{0,\gamma-s+a(t_0)\})\chi_{\{t_0\}}.$$

It is clear that $\tilde{a}, \tilde{b}$ are respectively lower and upper semicontinuous, hence there exists a continuous function $g: \ct \to \IR$ such that $b-g$ is bounded and $$\|\tilde{b}-g\|=\|g-\tilde{a}\|=\frac{1}{2}\|\tilde{b}-\tilde{a}\| = \frac{1}{2}(\tilde{b}(t_0)-\tilde{a}(t_0))$$
where the last equality is true because $\frac{1}{2}(\tilde{b}(t_0)-\tilde{a}(t_0)) = \max\{s-a(t_0),\gamma\}$. By the previous chain of equalities it must be $g(t_0) = \frac{1}{2}(\tilde{b}(t_0)+\tilde{a}(t_0)) = s$. On the other hand, it is clear that $$\max\{\|b-g\|,\|g-a\|\} \geq \max\{b(t_0)-g(t_0), g(t_0)-a(t_0), \frac{1}{2}\|b-a\|\},$$ let us see the reverse inequality.

\begin{itemize}
\item Given $t \in \ct \setminus \{t_0\}$, we have $\max\{|b(t)-g(t)|,|g(t)-a(t)|\} = \max\{|\tilde{b}(t)-g(t)|,|g(t)-\tilde{a}(t)|\} \leq \frac{1}{2}(\tilde{b}(t_0)-\tilde{a}(t_0)) = \max\{g(t_0)-a(t_0),\gamma\}$.
\item We also have $\max\{|b(t_0)-g(t_0)|,|g(t_0)-a(t_0)|\} = b(t_0)-g(t_0)$.
\end{itemize}

Hence, we deduce $$\max\{\|b-g\|,\|g-a\|\} \leq \max\{b(t_0)-g(t_0), g(t_0)-a(t_0), \frac{1}{2}\|b-a\|\}.$$

\end{proof}

\begin{theor}\label{ab}
Let $K$ be a compact space and ${(f_n)}_n$ a bounded sequence in $\cc(K)$. Def\mbox{}ine, for every $t \in K$,
  $$
  b(t) = \sup \{ \lim_{(n,\alpha)}
	    \sup_{\myatop{j\geq n}{\beta \geq \alpha}} f_j(t_\beta)
		    : {(t_\alpha)}_{\alpha \in \Lambda} \textrm{ is a net converging to } t \},
  $$
  $$
  a(t) = \inf \{ \lim_{(n,\alpha)}
	    \inf_{\myatop{j\geq n}{\beta \geq \alpha}} f_j(t_\beta)
		  : {(t_\alpha)}_{\alpha \in \Lambda} \textrm{ is a net converging to } t \},
  $$
where we consider the limits in the directed product set $\IN\times\Lambda$ with the product order. Then the functions $a$ and $b$ are lower and upper semicontinuous respectively, and for every $f \in \cc(K)$ we have
  $$
  \limsupline_n \|f_n - f\| = \max \{ \|b - f\|, \|f - a\| \}.
  $$
Consequently the asymptotic radius of ${(f_n)}_n$ is $\frac{1}{2}\|b-a\|$ and its center is the nonempty set $$\{g \in \cc(K): \|b-g\| = \|g-a\| = \frac{1}{2}\|b-a\| \}.$$
\end{theor}

\begin{proof}

Assume $b$ is not upper semicontinuous for some $t \in K$, and denote by $\cb$ the set of open neighbourhoods of $t$. Then there exists $\vareps > 0$ such that for each $V\in\cb$ there exists $t_V$ such that $b(t_V) > b(t) + \vareps$. By def\mbox{}inition of $b$, for each $V\in\cb$ there exists a net $t_{\alpha,V} \stackrel{\alpha}{\to} t_{V}$ satisfying $\disp \lim_{(n,\alpha)} \sup_{\myatop{j\geq n}{\beta \geq \alpha}} f_j(t_{\beta,V}) > b(t) + \vareps$.

Thus, we can obtain inf\mbox{}inite sets $M_V \subseteq \IN$ and $\{u_{m,V}: m \in M_V\} \subseteq V$ such that $f_m(u_{m,V}) > b(t) + \vareps$. For the sake of simplicity in the notation, def\mbox{}ine $u_{m,V} = t$ whenever $m \in \IN \setminus M_V$.

Consider the directed set $\Lambda = \IN \times \cb$ with the product order. It is clear that $u_{\alpha\in\Lambda} \to t$ and we also have (observe that the limit is taken in $\IN \times \Lambda = \IN^2 \times \cb$)
  $$
  \lim_{(n,\alpha)} \sup_{\myatop{j \geq n}{\beta\geq\alpha}} f_j(u_\beta) > b(t) + \vareps.
  $$

This contradiction proves that $b$ is upper semicontinuous. The proof for $a$ is entirely analogous.

Now def\mbox{}ine $u = \disp \limsupline_n \|f_n - f \|$, there exists a sequence ${(t_{n_j})}_j \subseteq K$ such that one of the following equalities holds:
\begin{itemize}
\item $\disp \lim_j (f_{n_j}(t_{n_j}) - f(t_{n_j})) = u$
\item $\disp \lim_j (f(t_{n_j}) - f_{n_j}(t_{n_j})) = u$
\end{itemize}
and, by compactness, there exists a cluster point of ${(t_{n_j})}_j $, call it $t$. Then again, one of the following equalities must hold:
\begin{itemize}
\item $\disp \limsupline_j (f_{n_j}(t_{n_j}) - f(t)) \geq u$
\item $\disp \limsupline_j (f(t) - f_{n_j}(t_{n_j})) \geq u$
\end{itemize}

In the f\mbox{}irst case we obtain $b(t) \geq f(t) + u$ and in the second, $a(t) \leq f(t) - u$. Considering both inequalities we arrive at
$$\max \{ \|b - f\|, \|f - a\| \} \geq \limsupline_n \|f_n - f\|.$$

For the converse inequality, f\mbox{}ix $t \in K$. Taking into account how $b$ is def\mbox{}ined, we have
$$a(t)-f(t)\leq b(t)-f(t) \leq \limsupline_n \|f_n-f\|$$
and analogously
$$f(t)-b(t) \leq f(t)-a(t) \leq \limsupline_n \|f_n-f\|.$$

Combining those inequalities we obtain $$\max \{ \|b - f\|, \|f - a\| \} \leq \limsupline_n \|f_n - f\|.$$ The rest of the proof is an immediate consequence of lemma \ref{ndist}.

\end{proof}

A similar version can be given in the case of $\cc_0(L)$ spaces, with only minor modif\mbox{}ications in the proof (it suf\mbox{}f\mbox{}ices to take $t_0 = \infty$ and $s=0$ in lemma \ref{ndist}):

\begin{theor}\label{ablocally}
Let $L$ be a locally compact, noncompact space and ${(f_n)}_n$ a bounded sequence in $\cc_0(L)$. Let $K$ be the one-point compactif\mbox{}ication of $L$, consider that each $f_n$ is def\mbox{}ined in $K$ by saying $f_n(\infty)=0$ and def\mbox{}ine $a, b :K \to \IR$ as in theorem \ref{ab}.

Then the functions $a$ and $b$ are lower and upper semicontinuous respectively, and for every $f \in \cc(K)$ we have
$$\limsupline_n \|f_n - f\| = \max \{ \|b - f\|, \|f - a\| \}.$$ Consequently the asymptotic radius of ${(f_n)}_n$ is $\max\{b(\infty),-a(\infty),\frac{1}{2}\|b-a\|\}$ and its center is the nonempty set $$\{g \in \cc_0(L): \max\{\|b-g\|,\|g-a\|\} = \max\{b(\infty),-a(\infty),\frac{1}{2}\|b-a\|\} \}.$$
\end{theor}

L. Vesel\'{y} (\cite{Ve}) proved that certain hyperplanes of $c_0$ are not $cc$. From this and the separability of $c_0$, it is not dif\mbox{}f\mbox{}icult to deduce (see \cite{LiRa}) that Vesel\'{y}'s examples are not $sacc$ either. Therefore, there are $2$-codimensional subspaces of $c$ which are not $sacc$. We do not know whether every $1$-codimensional subspace of a $\cc(K)$ space is $sacc$.

Next, we will apply the previous results to deduce Lim's expressions for the radii:

\begin{theor}[T. C. Lim, \cite{Lim83}]
Let $\sucx$ be a sequence in $c_0$, $c$ or $\ell_\infty$. Its asymptotic radius is, respectively:
  \begin{equation}\label{eq:limc0}
  \asra(\sucx)=\max\left\{\frac{1}{2} \lim_m \sup_k \left(\sup_{n\geq m} x_n(k)-\inf_{n\geq m} x_n(k)\right), \lim_m \limsupline_k \sup_{n\geq m}|x_n(k)|\right\},
  \end{equation}

  \begin{equation}\label{eq:limc}
  \asra(\sucx)=\frac{1}{2} \max\left\{\lim_m \sup_k \left(\sup_{n\geq m} x_n(k)-\inf_{n\geq m} x_n(k)\right), \lim_m (\limsupline_k \sup_{n\geq m} x_n(k)-\liminfline_k \inf_{n\geq m} x_n(k))\right\},
  \end{equation}

  \begin{equation}\label{eq:liml}
  \asra(\sucx)=\frac{1}{2} \lim_m \sup_k \left(\sup_{n \geq m} x_n(k) - \inf_{n \geq m} x_n(k)\right).
  \end{equation}
\end{theor}

\begin{proof}

Let $\sucx \in \ell_\infty$ and call
\begin{itemize}
\item For every $\mnat$, $\alpha_m = \disp \sup_k (\sup_{n\geq m} x_n(k)-\inf_{n\geq m} x_n(k))$.
\item $\alpha = \disp \lim_m \alpha_m = \inf_m \alpha_m$.
\item For every $\knat$, $\beta_k = \disp \limsupline_n x_n(k) - \liminfline_n x_n(k)$.
\item $\beta = \disp \sup_k \beta_k$.
\item $\gamma = \disp \inf_m \left(\sup_{\myatop{n\geq m}{k\geq m}} x_n(k)-\inf_{\myatop{n\geq m}{k\geq m}} x_n(k)\right) = \limsupline_{n,k} x_n(k) - \liminfline_{n,k} x_n(k)$.
\item $\delta = \disp \limsupline_{n,k}|x_n(k)|$.
\end{itemize}
Given $\vareps > 0$, for every $\mnat$ there exists $k_m \in \IN$ satisfying
  $$
  \sup_{n\geq m} x_n(k_m) - \inf_{n \geq m} x_n(k_m) > \alpha_m - \frac{\vareps}{2}\geq\alpha - \frac{\vareps}{2}.
  $$
Denote $F=\{k_m : \mnat\}$. On the one hand, if $F$ is f\mbox{}inite, then we have
  $$
  \inf_{\mnat} \sup_{k \in F}\left(\sup_{n\geq m} x_n(k) - \inf_{n \geq m} x_n(k)\right) \geq \alpha - \frac{\vareps}{2};
  $$
besides, for every $\knat$ there exists $m_k \in \IN$ such that 
  $$
  \sup_{n\geq m_k} x_n(k) - \inf_{n\geq m_k} x_n(k) \leq \beta_k + \frac{\vareps}{2}
  $$
which implies, if we take $m_0 = \max\{m_k : k \in F\}$, that 
  $$
  \sup_{k \in F}\left(\sup_{n\geq m_0} x_n(k) - \inf_{n\geq m_0} x_n(k) \right) \leq \beta + \frac{\vareps}{2}
  $$
and therefore $\alpha \leq \beta + \vareps$. On the other hand, if $F$ is inf\mbox{}inite then there exist two strictly increasing sequences ${(m_j)}_j$, ${(k_j)}_j$ such that 
  $$
  \alpha - \frac{\vareps}{2} \leq \inf_j \left(\sup_{n\geq m_j} x_n(k_j) - \inf_{n \geq m_j} x_n(k_j)\right) \leq \gamma \leq 2\delta.
  $$
Joining the two possibilities we deduce that $\alpha \leq \disp\max\{\beta, \gamma\}$. It is straightforward to see that $\alpha_m \geq \beta_k$ for every $m, k \in \IN$ and therefore $\beta \leq \alpha$. We deduce that 

\begin{equation} \label{eqc}
\max\{\beta, \gamma\} = \max\{\alpha, \gamma\}
\end{equation}
and

\begin{equation} \label{eqc0}
\max\{\beta, 2\delta\} = \max\{\alpha, 2\delta\}.
\end{equation}

In the case of $c$, we can identify $c$ with $\cc(\IN\cup\{\infty\})$ where $\IN\cup\{\infty\}$ is the one-point compactification of $\IN$. It is easy to see that Lim's expression equals $\frac{1}{2}\max\{\alpha, \gamma\}$ but applying Theorem~\ref{ab} we obtain that $\asra(\sucx)=\frac{1}{2}\|b-a\| = \frac{1}{2}\max\{\sup_{\knat} (b(k)-a(k)),b(\infty)-a(\infty)\} = \frac{1}{2}\max\{\beta,\limsupline_{n,k} x_n(k) - \liminfline_{n,k} x_n(k)\} = \frac{1}{2}\max\{\beta,\gamma\}$, so equation (\ref{eqc}) provides the desired equality~\eqref{eq:limc}.

In the case of $c_0$, seen as $\cc_0(\IN)$, it is easy to see that Lim's expression equals $\max\{\frac{1}{2}\alpha, \delta\}$ while by Theorem~\ref{ablocally} $\max\{\frac{1}{2}\|b-a\|,b(\infty),-a(\infty)\} = \max\{\frac{1}{2}\|b-a\|,\limsupline_{n,k} x_n(k),\limsupline_{n,k} -x_n(k)\} = \max\{\frac{1}{2} \beta ,\delta\}$, so equation (\ref{eqc0}) provides the desired equality~\eqref{eq:limc0}.

In the case of $\ell_\infty$ we identify this space with $\cc(\beta\IN)$ so we can apply Theorem~\ref{ab}. F\mbox{}ix $\vareps > 0$ and $t \in \beta\IN$. Call $\cv$ the set of neighbourhoods of $t$, for a given $\mnat$ and every $V \in \cv$ there exist $s_V, u_V \in \IN \cap V$ such that $\sup_{n\geq m} x_n(s_V) > b(t) - \vareps$ and $\inf_{n \geq m} x_n(u_V) < a(t) + \vareps$. We have that $t$ is a limit point of both ${(s_V)}_{V \in \cv}$ and ${(u_V)}_{V \in \cv}$. Since we are dealing with a Stone-\v{C}ech compactif\mbox{}ication, there must exist $k_m \in \IN$ which is both an $s_{V_1}$ and a $u_{V_2}$, thus having
  $$
  \sup_{n\geq m} x_n(k_m) > b(t) - \vareps \quad \textrm{ and } \quad \inf_{n \geq m} x_n(k_m) < a(t) + \vareps.
  $$
This implies, as $m$ was arbitrary, 
  $$
  \lim_n\sup_{\knat} \left(\sup_{n\geq m} x_n(k) - \inf_{n \geq m} x_n(k)\right) =\inf_{\mnat}\sup_{\knat} \left(\sup_{n\geq m} x_n(k) - \inf_{n \geq m} x_n(k)\right) \geq b(t)-a(t) - 2\vareps.
  $$
But this holds for every $\vareps>0$ and $t\in \beta \IN$, so by Theorem~\ref{ab} the right-hand side of equality~\eqref{eq:limc} is greater than or equal to the left-hand side so we have to prove the opposite inequality.

Now let $\vareps > 0$. For every $\mnat$ there exists $k_m$ such that 
  $$
  \sup_{n\geq m} x_n(k_m) - \inf_{n\geq m} x_n(k_m) \geq \alpha_m - \vareps
  $$ 
and if we consider the sequence $(k_1,k_2,\dots)$ it must have a subnet ${(t_r)}_{r \in \Lambda} \subseteq \IN$ converging to certain $t_0 \in \beta \IN$. This implies that there exist $n_0 \in \IN$ and $r_0 \in \Lambda$ satisfying 
  $$
  \|b-a\| \geq b(t_0) - a(t_0) \geq \sup_{\myatop{j \geq n_0}{r \geq r_0}} x_j(t_r) - \inf_{\myatop{j \geq n_0}{r \geq r_0}} x_j(t_r) - \vareps
  $$ 
but there exist $r_1 \geq r_0$ and $m_0 \geq n_0$ with $t_{r_1} = k_{m_0}$, yielding
  $$
  \sup_{\myatop{j \geq n_0}{r \geq r_0}} x_j(t_r) - \inf_{\myatop{j \geq n_0}{r \geq r_0}} x_j(t_r) - \vareps \geq \sup_{j \geq m_0} x_j(k_{m_0}) - \inf_{j \geq m_0} x_j(k_{m_0}) - \vareps \geq \alpha_{m_0} - 2\vareps \geq \alpha - 2 \vareps.
  $$
Again $\vareps$ was arbitrary and we arrive at the opposite inequality.
\end{proof}

To f\mbox{}inish this section, let us mention that there is no known formula for the radius and center in $\ell_1$, and it is also unknown whether $L_1[0,1]$ is asymptotically center-complete (both are stated as open problems in \cite{Lim83}). In $\cite{Lim80}$ it was proved that $\ell_1$ is asymptotically center-complete. The center completeness of $L_1[0,1]$ was proved in \cite{Ga}.

\section{Continuity properties of the asymptotic center}

As we mentioned in the introduction, several authors have studied the continuity properties of the center map in center-complete spaces, with respect to the Hausdorf\mbox{}f metric and frequently using Michael's theorem (\cite{Mi}) to obtain a continuous selector. Perhaps a good starting point for the interested reader would be the paper by D. Amir and J. Mach (\cite{AmMa}), which is a very well-written and detailed account. Here we will try to study the corresponding sequential properties; for this purpose, f\mbox{}irst we introduce an analogous of the Hausdorf\mbox{}f metric which seems suitable for sequences.

Given a sequence $\sucx = {(x_n)}_n$, we will write its $n$-th tail as $$C_n(\sucx) = \{x_m: m\geq n\}$$ and by means of the tails we can def\mbox{}ine a pseudometric in $\ell_\infty(X)$:
\begin{equation*}
\begin{split}
  \disp d(\sucx, \sucy) = \inf \{\vareps>0 : &\textrm{ given } \nnat \textrm{ there exists } \mnat \textrm{ such that }  \\
  &C_m(\sucx) \subseteq C_n(\sucy) + \vareps B_X \textrm{ and } C_m(\sucy) \subseteq C_n(\sucx) + \vareps B_X\}.
\end{split}
\end{equation*}

We will say that $\sucx \sim \sucy$ whenever $d(\sucx, \sucy)=0$, and accordingly def\mbox{}ine $Y=\ell_\infty(X) / \sim$. As usual, elements of $Y$ will be denoted by any class representative, i. e. $[\sucx]$. $Y$ is a metric space with the distance $d([\sucx],[\sucy])=d(\sucx,\sucy)$.

\begin{propo}
$\left(\ell_\infty(X) / \sim,d\right)$ is a complete metric space.
\end{propo}

\begin{proof}

Take $([\sucx_n])_n$, with each $\sucx_n={(x_n(s))}_s$ a bounded sequence in $X$, such that
  $$
  d(\sucx_n,\sucx_{n+1})<\frac{1}{2^n}
  $$
for all $n$. We have to prove that this sequence converges. Fix $n_1=1$. Since $d(\sucx_1,\sucx_{2})<\frac{1}{2}$
there is $n_2>n_1$ such that 
  $$
  C_{n_2}(\sucx_{1}) \subseteq C_{n_1}(\sucx_2) + \frac{1}{2} B_X\textrm{ and }C_{n_2}(\sucx_2) \subseteq C_{n_1}(\sucx_1) + \frac{1}{2} B_X.
  $$
Suppose that we have obtained $n_m$. Since
  $$
  d(\sucx_i,\sucx_{j})<\sum_{k=\min\{i,j\}}^{\max\{i,j\}-1}\frac{1}{2^k}
  $$
there is $n_{m+1}$ such that
  \begin{equation}\label{eq:relacioncn}
  C_{n_{m+1}}(\sucx_{i})\subseteq C_{n_m}(\sucx_j)+\left(\sum_{k=\min\{i,j\}}^{\max\{i,j\}-1}\frac{1}{2^k}\right)B_X\;\textrm{ for }\;1\leq i,j\leq m+1.
  \end{equation}
Then for each $m\in\IN$ we can choose a f\mbox{}inite set $A_m=\{y_{p_m+1},y_{p_m+2},\dots,y_{p_{m+1}}\}\subseteq C_{n_m}(\sucx_{m})$ such that
  \begin{equation}\label{eq:xsam}
  x_i(s) \in A_m+\left(\sum_{k=i}^{m-1}\frac{1}{2^k}\right)B_X\;\text{ for }\;i<m\;\text{ and }n_{m+1}\leq s<n_{m+2}.
  \end{equation}
Def\mbox{}ine $\sucy={(y_p)}_p$. We will prove that the sequence ${([\sucx_n])}_n$ converges to $[\sucy]$. For this we prove that $d(\sucx_i,\sucy)<\frac{1}{2^{i-1}}$.

Fix $n\in\IN$, and pick $t>i$ with $p_t+1>n$. If $s\geq n_{t+1}$ there is $m\geq t$ such that $n_{m+1}\leq s<n_{m+2}$ and then by~\eqref{eq:xsam} we have that
  $$
  x_i(s)\in A_m+\frac{1}{2^{i-1}}B_X\subseteq C_{p_m+1}(\sucy)+\frac{1}{2^{i-1}}B_X\subseteq C_{n}(\sucy)+\frac{1}{2^{i-1}}B_X
  $$
so 
  \begin{equation}\label{eq:completolado1}
  C_{n'}(\sucx_i)\subseteq C_{n}(\sucy)+\frac{1}{2^{i-1}}B_X\;\text{ if }\;n'\geq n_{t+1}.
  \end{equation}
Pick now $m$ such that $n_m\geq n$ and $m>i$ and take $l\in\IN$. By~\eqref{eq:relacioncn} we have that
  \begin{equation*}
  \begin{split}
  A_{m+l}&\subseteq C_{n_{m+l}}(\sucx_{m+l}) \subseteq C_{n_{m+l-1}}(\sucx_{m+l-1})+\frac{1}{2^{m+l-1}}B_X\subseteq\cdots\subseteq\\
  &\subseteq C_{n_{m+1}}(\sucx_{m+1})+\left(\sum_{k=m+1}^{m+l-1}\frac{1}{2^k}\right)B_X
    \subseteq C_{n_{m}}(\sucx_{i})+\left(\sum_{k=i}^{m+l-1}\frac{1}{2^k}\right)B_X\subseteq\\
    &\subseteq C_n(\sucx_i)+\frac{1}{2^{i-1}}B_X.
  \end{split}
  \end{equation*}
Thus
  \begin{equation}\label{eq:completolado2}
  C_{n'}(\sucy)\subseteq C_{n}(\sucx_i)+\frac{1}{2^{i-1}}B_X\;\text{ if }\;n'> p_{m+1}.
  \end{equation}
Combining~\eqref{eq:completolado1} and~\eqref{eq:completolado2} we get that $d(\sucx_i,\sucy)<\frac{1}{2^{i-1}}$.
\end{proof}

In the proposition and conjecture that follows we try to advocate that this distance is, in certain sense, ``sharp'' concerning centers.

\begin{propo}
Let $\sucx, \sucy$ be bounded sequences in a Banach space $X$ satisfying $d(\sucx,\sucy)=0$. Then
\begin{enumerate}
\item $\asce(\sucx)=\asce(\sucy)$ and $\asra(\sucx)=\asra(\sucy)$.
\item $d(\sucx,\sucy)=0$ for each equivalent renorming of $X$.
\end{enumerate}
\end{propo}

\begin{proof}

Take $z \in X$. For every $\vareps > 0$ we have $d(\sucx,\sucy) < \vareps$ and this implies $|\limsupline_n \|x_n-z\| - \limsupline_n \|y_n-z\|| < \vareps$. Since $\vareps$ is arbitrary we deduce that $$\limsupline_n \|x_n-z\| = \limsupline_n \|y_n-z\|$$ which yields immediately $\asra(\sucx)=\asra(\sucy)$ and $\asce(\sucx)=\asce(\sucy)$.

The second statement is a direct consequence of the def\mbox{}inition of $d$.
\end{proof}

Is there a sort of converse to the previous proposition?

\begin{conje}
Let $\sucx,\sucy$ be bounded sequences in a Banach space $X$. If $\asce(\sucx)=\asce(\sucy)$ and $\asra(\sucx)=\asra(\sucy)$ hold for each equivalent renorming of $X$, then $d(\sucx,\sucy) = 0$.
\end{conje}

The condition ``for each equivalent renorming'' cannot be removed from the conjecture. Indeed, in the euclidean $\IR^2$ consider the sequences ${((-1)^n,0)}_n$ and ${(0,(-1)^n)}_n$. Their distance is $\sqrt{2}$ but they both have asymptotic center $\{0\}$ and asymptotic radius $1$. Note that their asymptotic centers are no longer the same if we choose, e. g., the sup norm.

J. Mach (\cite{Ma}, p. $225$) introduced a property called $P_2$ to prove the existence of continuous selectors for the center map. The following notion of continuity serves the analogous purpose for sequential asymptotic centers:

\begin{defi}
Let $X$ be a Banach space. We will say that $X$ has {\bf continuity with respect to asymptotic centers} (in short, {\bf $\mach$}) if there exists $\delta > 0$ such that every bounded sequence $\sucx \subseteq X$ satisf\mbox{}ies $$\ascd{\delta}(\sucx)\subseteq B_X+\asce(\sucx).$$ If we want to be more specif\mbox{}ic we will say that the space has {\bf $\dmach$}.
\end{defi}

The following theorem can be applied to all pseudometrics sharing a certain feature of $d$.

\begin{theor}\label{selectors}
Let $X$ be a Banach space and $\rho: \ell_\infty(X) \to \IR$ be a pseudometric such that 
\begin{itemize}
\item $\rho(\sucx,\sucy)=0$ implies $\asra(\sucx)=\asra(\sucy)$ and $\asce(\sucx)=\asce(\sucy)$.
\end{itemize}
If $X$ has $\dmach$ then the multivalued mapping $T: \ell_\infty(X) / \rho \longto 2^X$ given by $T([\sucx])=\asce(\sucx)$ satisf\mbox{}ies:
\begin{itemize}
\item Every $T([\sucx])$ is convex, closed and nonempty.
\item $T$ is lower semicontinuous.
\end{itemize}
In other words, $T$ is in the situation of Michael's selection theorem and thus it has a continuous selector.
\end{theor}
		
\begin{proof}

Note that we only need to prove that 
  $$
  W:=\{\sucx \in \ell_\infty(X) : \asce(\sucx) \cap U_X \neq \varnothing\}
  $$
is open, where $U_X$ is the open unit ball of $X$. Assume that $\sucx \in W$ and take $u \in \asce(\sucx) \cap U_X$ and $\vareps > 0$ such that $B(u, \vareps) \subseteq U_X$.

If $\sucy$ satisf\mbox{}ies $\|\sucx-\sucy\| < \delta\vareps/2$ then it is easy to see that $|\asra(\sucx)-\asra(\sucy)|<\delta\vareps/2$ and thus 
  $$
  \limsupline_n \|y_n-u\| < \frac{\delta\vareps}{2} + \asra(\sucx) < \delta\vareps + \asra(\sucy).
  $$
If we take $z_n=\vareps^{-1}y_n$ then $\asra(\sucz)=\vareps^{-1}\asra(\sucy)$ and the previous inequality implies that 
  $$
  \vareps^{-1}u \in \ascd{\delta}(\sucz).
  $$
Consequently $\vareps^{-1}u \in B_X + \asce(\sucz)$, which in turn leads to $u \in \vareps B_X + \asce(\sucy)$.

We deduce that there exists $v \in \asce(\sucy)$ with $\|u-v\|\leq \vareps$ and so $v \in U_X$. We conclude that $\sucy \in W$ and then $W$ is a open set.

\end{proof}

Next corollary follows from previous theorem when $\rho=d$.

\begin{coroll}\label{corollphi}
If a Banach space $X$ has $\dmach$ then there exists $\varphi: \ell_\infty(X) \to X$ continuous such that $\varphi(\sucx) \in \asce(\sucx)$, and $\varphi(\sucx)=\varphi(\sucy)$ whenever $d(\sucx,\sucy)=0$. In particular $\varphi$ satisf\mbox{}ies:
\begin{itemize}
\item $\varphi(\sucx)= \lim_n x_n$ if $\sucx$ converges.
\item $\varphi(\sucx) = \varphi(F(\sucx))$ where $F$ is the forward operator.
\item $\varphi(\sucx) = \varphi({(x_{\pi(n)})}_n)$ for every bijection $\pi: \IN \to \IN$.
\end{itemize}
\end{coroll}

It is not dif\mbox{}f\mbox{}icult to see that $\varphi$ cannot be additive even in the simplest space $X = \IR$. However, it would be interesting to study whether $\varphi(\sucx+\sucy) = \varphi(\sucx)+\varphi(\sucy)$ given that $\sucx$ is arbitrary and $\sucy$ is convergent. Clearly, this holds if asymptotic centers are always unitary in the space.

Which spaces have $\mach$? At least, certain well-placed subspaces of the $\cc(K)$ spaces:

\begin{theor}
Let $K$ be a Hausdorf\mbox{}f, compact space and $Y \subseteq \cc(K)$ a closed subspace. If $Y$ has the properties:
\begin{enumerate}
\item $Y$ is sequentially asymptotically center-complete.
\item There exists $\delta > 0$ such that for every $f \in Y$ there exists $u:\IR \to [-1,1]$ satisfying
\begin{itemize}
\item $u(0)=0$.
\item $\min\{1,\frac{\delta}{|x|}\} \leq \frac{u(x)}{x} \leq 1$ \quad if $ x \neq 0$.
\item $u \circ f \in Y$
\end{itemize}
\end{enumerate}
then $Y$ has $\dmach$.
\end{theor}

\begin{proof}

Let ${(f_n)}_n \subseteq Y$ be a bounded sequence and call $r=\asra({(f_n)}_n)$. By virtue of theorem \ref{ab}, there exist $a, b: K \to \IR$, lower and upper semicontinuous respectively, such that
  $$
  \asce({(f_n)}_n) = \{g \in Y: \max \{ \|b - g\|, \|g - a\| \} = r\}
  $$
and 
  $$
  \ascd{\delta}({(f_n)}_n) = \{h \in Y: \max \{ \|b - h\|, \|h - a\| \} \leq r + \delta \}.
  $$

Fix $g \in \asce({(f_n)}_n)$. Given $h \in \ascd{\delta}({(f_n)}_n)$, for $g-h\in Y$ consider $u:\IR \to [-1,1]$ as in the hypothesis. Let us see that $z = h + u \circ(g-h)\in \asce({(f_n)}_n)$. We have $z \in Y$ and $\|z-h\| = \|u \circ(g-h)\| \leq 1$. For every $x \in X$:

\begin{itemize}
\item If $(g-h)(x) = 0$ then $z(x)=g(x)$.
\item If $(g-h)(x) > 0$ then $\min\{(g-h)(x),\delta\} \leq u((g-h)(x))=(z-h)(x) \leq (g-h)(x)$, which implies 
  \begin{equation*}\begin{split}
  -r &\leq \min\{g(x)-a(x),h(x)-a(x)+\delta\} \leq \min\{g(x)-h(x),\delta\}+h(x)-a(x) \leq \\
  &\leq z(x)-a(x)\leq g(x)-a(x) \leq r.
  \end{split}\end{equation*}
\item If $(g-h)(x) < 0$ then $(g-h)(x) \leq (z-h)(x) = u((g-h)(x)) \leq \max\{(g-h)(x),-\delta\}$, which implies 
  \begin{equation*}\begin{split}
  -r &\leq g(x)-a(x) \leq z(x)-a(x)\leq \max\{(g-h)(x),-\delta\}+h(x)-a(x) \leq \\
  &\leq \max\{g(x)-a(x),h(x)-a(x)-\delta\} \leq r.
  \end{split}
  \end{equation*}
\end{itemize}

We deduce that $|z(x)-a(x)| \leq r$. Proceeding in the same way with $|z(x)-b(x)|$, it is now clear that 
  $$
  \max\{\|z-a\|,\|z-b\|\} \leq r.
  $$
Therefore $z\in \asce({(f_n)}_n)$ and then $h\in B_X+\asce({(f_n)}_n)$.

\end{proof}

Clearly, the second condition in the previous theorem might be hard to check in some subspaces. Nonetheless, it is straightforward to see that this condition is satisf\mbox{}ied by every subspace of $\cc(K)$ that contains the constants and is closed under taking absolute value.

\subsection{Hilbert spaces}

Here we will prove more than just the continuity, showing that, in the case of Hilbert spaces, a sort of H\"older condition for the (uniquely def\mbox{}ined) selector can be obtained. We are based in \cite{BaPa}, where M. Baronti and P. L. Papini proved the following result concerning centers in a Hilbert space:
$$\|c(A)-c(B)\|^2 \leq d_H(A,B)(r(A)+r(B)+d_H(A,B))$$
where $A$ and $B$ are bounded sets and $d_H$ is the Hausdorf\mbox{}f metric.

Let us see that the analogous result holds for bounded sequences and asymptotic centers. What follows are suitable modif\mbox{}ications of proposition $2.3$ and corollary $2.5$ in \cite{AmMa} which seem to f\mbox{}it our purpose. This will be achieved in theorem \ref{seqbapa}, whose proof uses essentially the techniques of \cite{BaPa} with some necessary adjustments. Although it is not strictly necessary in the sequel, let us recall that the asymptotic center in Hilbert spaces is always a unitary set (\cite{Ed}).

\begin{lemma}\label{centerinco}
Let $X$ be a Banach space and $\sucx = {(x_n)}_n$ a bounded sequence in $X$, with asymptotic radius $r$ and having $z$ as an asymptotic center. For every $\vareps > 0$, consider the subsequence ${(v_n)}_n$ of $\sucx$ determined by the inf\mbox{}inite set $V = \{\nnat : \|x_n-z\| > r - \vareps\}$. Then $z$ is an asymptotic center and $r$ is the asymptotic radius of ${(v_n)}_n$.
\end{lemma}

\begin{proof}

Assume the asymptotic radius of ${(v_n)}$ is smaller than $r$. Then there exists $y$ satisfying $\disp \limsupline_{n \in V} \|x_n - y\| < r$. Take $\mu \in (0,1)$ with $\mu\|y-z\| < \vareps$ and def\mbox{}ine $z_0 = z + \mu (y-z)$. We have

  $$\limsupline_{n \in \IN \setminus V} \|x_n - z_0\| \leq \|z_0 - z\| + \limsupline_{n \in \IN \setminus V} \|x_n-z\| \leq \mu \|y-z\| + r - \vareps < r
  $$
and
  $$
  \limsupline_{n \in V} \|x_n - z_0\| \leq (1-\mu)\limsupline_{n \in V} \|x_n-z\| + \mu \limsupline_{n \in V} \|x_n - y\| < (1-\mu)\cdot r+\mu\cdot r= r,
  $$
which proves that the asymptotic center of ${(x_n)}_n$ is smaller than $r$, a contradiction. Therefore, $r$ is the asymptotic radius of ${(v_n)}_n$, which also has $z$ as an asymptotic center.
\end{proof}

\begin{lemma}\label{ammabapa}
Let $X$ be a Hilbert space and $\sucx ={(x_n)}_n$ a bounded sequence in $X$, with asymptotic center $z$ and asymptotic radius $r$. Then
  $$
  z \in \bigcap_{\myatop{\knat}{\vareps > 0}} \adh{\co}\left(C_k(\sucx) \setminus B\left(z, r - \vareps \right) \right).
  $$

In particular, for every $\vareps > 0$ and $x^* \in X^*$ there exists a subsequence ${(u_n)}_n$ of $\sucx$ satisfying, for every $\nnat$,
\begin{itemize}
\item $\|u_n-z\| \geq r - \vareps$,
\item $x^*(u_n)  \geq x^*(z) - \vareps$.
\end{itemize}

\end{lemma}

\begin{proof}

We will write $C_k$ instead of $C_k(\sucx)$. Call $A = \bigcap_{\knat} \adh{\co}(C_k)$, and let us see f\mbox{}irst the weaker statement $z \in A$. Assume on the contrary that $z \notin A$, and let $y$ be the projection of $z$ in the convex, closed set $A$. It is well known that for every $a \in A$ one has $(z-y|a-y) \leq 0$, where $(\ |\ )$ denotes the inner product of $X$. If we consider $\alpha = (z-y|y)$ and $f: X \to \IR$ given by $f(b) = (z-y|b)$, we have that $f \in X^*$ and $A \subseteq f^{-1}((-\infty,\alpha])$.

Let us see that for every $\mu > 0$, the set $\{n \in \IN: x_n \in f^{-1}((-\infty,\alpha+\mu])\}$ is cof\mbox{}inite. Assume the opposite, then there exists a subsequence ${(x_{n_k})}_{\knat}$ of $\sucx$ satisfying $f(x_{n_k}) > \alpha + \mu$ if $k \in \IN$. We know by ref\mbox{}lexivity that there exists $b \in \bigcap_{\knat} \adh{\co}(\{x_{n_j}: j \geq k\})$ and then $f(b) \geq \alpha + \mu$. However, $\{x_{n_j}: j \geq k\} \subseteq C_{n_k}$ and therefore $b \in \bigcap_{\knat} \adh{\co}(C_{n_k}) = \bigcap_{\knat} \adh{\co}(C_k) = A$, which is a contradiction.

Now take $\mu = \frac{\|y-z\|^2}{4}$ and consider $n_0 \in \IN$ such that if $n \geq n_0$ then $f(x_n) \leq \alpha + \mu$. Then we have
  $$
  \|z-x_n\|^2 = \|y-x_n\|^2+\|z-y\|^2+2(y-x_n|z-y) \geq \|y-x_n\|^2 + \|z-y\|^2 - 2\mu = \|y-x_n\|^2 + \frac{\|z-y\|^2}{2}
  $$
and this proves that $\limsupline_n \|z-x_n\| > \limsupline_n \|y-x_n\|$, again a contradiction. Thus we obtain $z \in A$.

Given $\vareps > 0$, the sequence ${(v_n)}_n$ def\mbox{}ined as in lemma \ref{centerinco} also has asymptotic center $z$ and asymptotic radius $r$. Reasoning as in the previous paragraphs we get to $$z \in \bigcap_{\knat} \adh{\co}(\{v_n: n \geq k\}) \subseteq \bigcap_{\knat} \adh{\co}(C_k \setminus B(z, r-\vareps)).$$

To conclude, f\mbox{}ix $\vareps > 0$ and $x^* \in X^*$. Given $\knat$, we have $$x^*(z) \leq \sup\{x^*(y): y \in C_k \setminus B\left(z, r - \vareps \right)\}.$$ Now a simple inductive process can be used to build the sequence ${(u_n)}_n$: just choose appropriate elements in $C_k \setminus B(z,r-\vareps)$, with $k$ increasing as necessary.

\end{proof}

\begin{theor}\label{seqbapa}
Let $X$ be a Hilbert space and $\sucx={(x_n)}_n, \sucy={(y_n)}_n$ bounded sequences in $X$. We have
  $$
  \|\asce(\sucx)-\asce(\sucy)\|^2 \leq d(\sucx,\sucy)(\asra(\sucx)+\asra(\sucy)+d(\sucx,\sucy)).
  $$
\end{theor}

\begin{proof}

For the sake of abbreviation, let us write $c_1 = \asce(\sucx)$, $c_2 = \asce(\sucy)$, $r_1 = \asra(\sucx)$, $r_2 = \asra(\sucy)$ and $d = d(\sucx,\sucy)$. Now consider $x^* \in X^*$ given by $x^*(v)=(v|2(c_1-c_2))$. By using lemma \ref{ammabapa} applied to such $x^*$, it is straightforward to deduce that for every $\vareps > 0$ there exists a subsequence of $\sucx$, say ${(u_n)}_n$, such that
\begin{equation}\label{equ1}
\limsupline_n \|u_n - c_1\| \geq r_1 - \vareps
\end{equation}
and, for every $\nnat$,
\begin{equation}\label{equ2}
\|u_n - c_2\|^2 \geq \|u_n - c_1\|^2 + \|c_1 - c_2\|^2 - \vareps.
\end{equation}
Equations (\ref{equ1}) and (\ref{equ2}) combined produce
$$\limsupline_n \|u_n - c_2\|^2 \geq \left(r_1-\vareps \right)^2 + \|c_1-c_2\|^2 - \vareps.$$ On the other hand, it is clear that $$\limsupline_n \|u_n - c_2\| \leq \limsupline_n \|x_n-c_2\| \leq r_2 + d,$$ which implies $$\|c_1-c_2\|^2 \leq (r_2+d)^2-\left(r_1-\vareps\right)^2 + \vareps.$$ Since this happens for every $\vareps > 0$, we have $$\|c_1-c_2\|^2 \leq (r_2+d)^2 - r_1^2$$ and symmetrically we also have $$\|c_1-c_2\|^2 \leq (r_1+d)^2 - r_2^2,$$ joining both assertions yields the desired inequality.
\end{proof}

{\bf Acknowledgements.} The authors would like to thank Bernardo Cascales for his many suggestions that have greatly improved this paper. The second and third author would like to thank Carlos Angosto and Bernardo Cascales for their warm hospitality during a short stay at University of Murcia.


\end{document}